\documentclass[12pt]{article}
\usepackage{amsmath}
\usepackage{amssymb}

\newlength{\defbaselineskip}
\newcommand{\setlinespacing}[1]%
           {\setlength{\baselineskip}{#1 \defbaselineskip}}
\newcommand{\doublespacing}{\setlength{\baselineskip}%
                           {2.0 \defbaselineskip}}

\newtheorem {mythm}{Theorem}  
\newtheorem {mylem}{Lemma}
\newtheorem {mycor}{Corollary}
\newtheorem {ex}{Example}
\newtheorem {mydef}{Definition}

\newcommand{\QED}{\hspace*{\fill}\rule{2.5mm}{2.5mm}}

\newenvironment{proof}{\noindent{\bf Proof.\ }}{\QED\\}

\newcommand{\ds}{\displaystyle}

\newcommand{\C}{\mathbb C}

\def \proclaim#1{\smallskip\noindent{\bf\ignorespaces#1\unskip.}%
 \bgroup\it\space\ignorespaces}
\def \endproclaim{\par\egroup\smallskip}

\begin{document}


\title{\bf More symmetric polynomials related to $p$-norms.}

\date{}
\author{Ivo Kleme\v{s} }
\maketitle
{\centerline
{\it
\noindent Department of Mathematics and Statistics,
\noindent 805 Sherbrooke Street West,} }

{\centerline
{\it
\noindent McGill University,
\noindent Montr\'eal, Qu\'ebec,
\noindent H3A 0B9,
\noindent Canada.
}}
\medskip

{\centerline{  Email:  klemes@math.mcgill.ca }}

\bigskip

\bigskip

\noindent {\it Abstract.} It is known that the elementary
symmetric polynomials $e_k(x)$ have the property that if
$ x, y \in [0,\infty)^n$ and
$e_k(x) \leq e_k(y)$ for all $k$, then $||x||_p \leq ||y||_p$
for all real $0\leq p \leq 1$, and moreover $||x||_p \geq ||y||_p$
for $1\leq p \leq 2$ provided $||x||_1 =||y||_1$. Previously
the author proved this kind of property for $p>2$, for certain polynomials
$F_{k,r}(x)$ which generalize the $e_k(x)$. In this paper we
give two additional generalizations of this type, involving
two other families of polynomials. When $x$ consists of the eigenvalues
of a matrix $A$, we give a formula for the polynomials in
terms of the entries of $A$, generalizing sums of
principal $k \times k$ subdeterminants.


\vfill
\noindent {\it A.M.S. Mathematics Subject Classifications:} 47A30 (05E05, 15A15).
\bigskip


\bigskip

\noindent  {\it Date:} 17 February 2013.

\newpage
%
%


\doublespacing

 \centerline{\bf \S 1. Introduction}
\medskip

Let $P_r(s)= 1 + \frac{s^1}{1!} + \dots + \frac{s^r}{r!}$, the $r$th Taylor polynomial
of the exponential function $e^s$. Let
 $F_{k,r}(x)$ denote the coefficient of $t^k$ in the product
\begin{equation}
\label{gen1}
\prod_{i=1}^n \ P_r(x_it) =
\prod_{i=1}^n \left(1 + \frac{(x_it)^1}{1!} + \dots + \frac{(x_it)^r}{r!}\right)
=1+\sum_{k=1}^{nr} F_{k,r}(x)t^k,
\end{equation}
where $x:=(x_1,\dots,x_n)$ for some $n$. For example, when $r=1$ we have
the elementary symmetric polynomials $e_k(x)=F_{k,1}(x)$ in $n$ variables.
In \cite{k2} it was shown that the $F_{k,r}(x)$ can be used to obtain
inequalities for the $p$-norms $||x||_p :=
(\frac{1}{n}\sum_{i=1}^n x_i^p)^{1/p}$ in certain intervals of the real number $p$
in the following sense:

\noindent {\bf Theorem A.}
{\it Let $x,y \in [0,\infty)^n$ and fix an integer $r \geq 1. $
Suppose that
\begin{equation}
\label{F0}
 F_{k,r}(x) \leq  F_{k,r}(y)
 \end{equation}
for all integers  $k$
in the interval $r \leq k \leq nr$. Then
\begin{equation}
\label{p0}
 ||x||_p \leq ||y||_p \ \ \ {\it whenever} \ \ \ 0 \leq p \leq 1.
\end{equation}
If also $ {\displaystyle \sum_{i=1}^n x_i  = \sum_{i=1}^n y_i  }$, then
\begin{equation}
\label{p1}
 ||x||_p \geq ||y||_p \ \ \ {\it whenever} \ \ \ 1 \leq p \leq r+1.
\end{equation}
}\\
%

\noindent
(By continuity in $p$, the $0$-norm is defined to be the geometric mean;
$||x||_0:=(\prod x_i )^{1/n}$.)
The $F_{k,r}(x)$, together with Theorem A, may be viewed as
one possible way to generalize the well-known case
$r=1$ of the theorem \cite[Ch. 4, p. 211-212, Lemma 11.1]{GK},
 which only gives information in the range $p<2$ and uses only
the elementary symmetric polynomials. The
purpose of this note is to give two other generalizations of the case
$r=1$ having the same kinds of
conclusions as Theorem A in the range $p>2$, but using two new families of
symmetric polynomials, different from the above $F_{k,r}(x)$.
One reason for
seeking such results in the range $p>2$ is a certain
open problem on the $p$-norms of the eigenvalues
$\{x_i\}$ of a matrix $A=QQ^*$ where $Q$ is a
$(0,1)$ ``interval matrix". For this motivation we refer the reader to
\cite[Theorem 1.2]{k1} and \cite[Example 1]{k2}. In this connection,
one additional feature of the new polynomials is that they obey
certain identities in terms of the power sum polynomials $p_m(x)
= \sum_i x_i^m \ (m=1,2, \dots)$, similar to the Newton and
``cycle index" identities
for elementary symmetric polynomials. When
the $x_i$ are the eigenvalues of a matrix $A$
these identities lead, via
the Kronecker power $A^{\otimes k}$, to certain expressions
for the new polynomials in terms of the entries of $A$. These
expressions generalize the formula for $e_k(x)$ as the
sum of principal $k \times k$ subdeterminants of $A$
(see (\ref{gA}) to (\ref{Dr}) in \S 6).


 \centerline{\bf \S 2. Backgroud to Theorem A.}
\medskip

We begin with a review of the proof of (\ref{p1})
in the basic case $r=1$ in Theorem A. There is an
integral formula (Mellin transform) for the power $a^p$ of a positive
real number $a$: For any ``suitable" function $\psi(t)$, it is easily seen
that
\begin{equation}
\label{r1}
a^p =
\frac{1}{C_p(\psi)}\int_0^{\infty} \psi(at) t^{-p}\frac{dt}{t},
\quad  {\rm where} \quad
C_p(\psi) = \int_0^{\infty} \psi(t) t^{-p}\frac{dt}{t}.
\end{equation}
\medskip
For $\psi$ to be ``suitable", we mean that the above improper integral $C_p(\psi)$
should converge and be nonzero. For example, with $\psi(t)= t-\log(1 + t) \geq 0$,
the integrals converge for $1<p<2$, and we have $C_p(\psi)>0$.
The restriction $p<2$ is due to the requirement that the integrals (\ref{r1})
converge as $t\to 0^+$; one sees that $t-\log(1 + t)$ decays like $t^2$.
Similarly, as $t\to \infty$, the $t-\log(1 + t)$ grows like $t^1$, so that
the restriction $p>1$ is needed. By applying the formula with $a=x_i$
and $a=y_i$ and summing over $i$, one sees that for
the case $r=1$ of (\ref{p1}) it is sufficient to
assume
$$\sum_i \bigg(x_it -\log(1 + x_it)\bigg)
\geq \sum_i \bigg(y_it -\log(1 + y_it)\bigg)$$
for all $t>0$. For the latter, it is in turn sufficient to have
the hypotheses in Theorem A; that $\sum x_i = \sum y_i$ and
that each coefficient in $\prod_i(1 + x_it)$ increases when $x$ is replaced
by $y$, which is exactly the hypothesis that $F_{k,1}(x) \leq F_{k,1}(y)$
where the $F_{k,1}$ are the elementary symmetric polynomials.

Next, suppose that we want a formula for $a^p$ valid for some
$p >2$. We could attempt to replace $t-\log(1 + t)$ by a function which
decays faster as $t\to 0^+$. For example, the new function
 $\psi(t)=t-\log(1 + t + \frac{t^2}{2!})$ can be seen to decay like $t^3$,
and thus yields a formula for $a^p$ in the range $1<p<3$. Extending this pattern, one sees
that for any positive integer $r$ the function
$\psi(t)=t-\log(1 + t + \frac{t^2}{2!}+\dots + \frac{t^r}{r!})$ is positive and
decays like $t^{r+1}$, and thus gives $a^p$ in the range
$1<p<r+1$. Unravelling the required inequalities,
to obtain the result (\ref{p1}) it is clearly sufficient to
have $\sum x_i = \sum y_i$ and the polynomial inequalities
$F_{k,r}(x)\leq F_{k,r}(y)$, as in Theorem A.

There are other ways to modify the function
$\psi(t)=t-\log(1 + t)$ to make it decay faster than $t^2$ as $t\to 0^+$
 (while preserving some other useful aspects of the above proof).
 We now look at two other such modifications, thereby obtaining,
 after some manipulation,
  two new families of symmetric polynomials which can be
 used instead of the $F_{k,r}(x)$ in an analogous manner.
\medskip

\centerline{\bf \S 3. The Polynomials $G_{k,r}(x)$.}
\medskip
Here the idea will be that instead of modifying $\log(1+t)$ from the ``inside"
as was done above to obtain Theorem A, we now modify it from the ``outside"
and see what is obtained.
Thus, we will subtract the Taylor polynomial of $\log(1+t)$ of some given
degree $r$.
\begin{mylem}
Fix an integer $r\geq 0$ and let $Q_r(t):=t-\frac{1}{2}t^2+\dots +(-1)^{r-1}\frac{1}{r}t^r$, the
$r$th Taylor polynomial of $\log(1 + t)$. (For $r=0$ define $Q_0(t):=0$.)
Define
\begin{equation}
\label{psi1}
\psi_r(t) := (-1)^r\bigg(\log(1 + t)- Q_r(t)\bigg).
\end{equation}
 Then $\psi_r(t) > 0$ when
$t>0$, and $\psi_r(t) = \mathcal{O}(t^{r+1})$
as $t\to 0^+$. Also, for $r \geq 1$ we have
$\psi_r(t) = \mathcal{O}(t^r)$ as $t\to \infty$, and for all
$\epsilon > 0$ we have
$\psi_0(t) = \mathcal{O}(t^\epsilon)$ as $t\to \infty$.
\end{mylem}
\begin{proof} The proof is a standard exercise in calculus:
Differentiating and then using the formula for the sum of a
geometric series we obtain\\

${\ds \frac{d}{dt}\psi_r(t) =
(-1)^r\bigg(\frac{1}{1+t}- 1+t-t^2+ \dots + (-1)^{r-1}t^{r-1}\bigg)}$\\

${\ds=(-1)^r\frac{(-t)^{r}}{1+t} = \frac{t^{r}}{1+t}.}$\\

\noindent
This shows that $\psi_r'(t) > 0$ for $t>0$, and $\psi_r'(t)=\mathcal{O}(t^{r})$
as $t\to 0^+$.
Since $\psi_r(0) =0$, the Mean Value Theorem implies that
$\psi_r(t) > 0$ when
$t>0$ and that $\psi_r(t) = \mathcal{O}(t^{r+1})$ as $t\to 0^+$.
Finally, the assertions for the case $t\to \infty$ follow from the
fact that the polynomial $Q_r(t)$ is of degree $r$, and from the
growth properties of the logarithm.
\end{proof}

\noindent
It follows that $a^p$ can be represented by (\ref{r1}) using
$\psi=\psi_r$ whenever $p$ is in the interval $r<p<r+1$,
where the corresponding constant $C_p(\psi_r)$ is positive.
Hence, we deduce:
\begin{mylem}
Let $x,y \in [0,\infty)^n$ and fix an integer $r\geq 0$.
If
\begin{equation}
\label{lem2}
\sum_{i=1}^n \psi_r(x_it)
\geq
\sum_{i=1}^n \psi_r(y_it)
 \end{equation}
for all $t>0$, then
$||x||_p \geq ||y||_p$ for all $p$ in the interval $ r \leq p \leq r+1$.
\end{mylem}
Assume now the additional hypothesis $\sum_i x_i =\sum_i y_i$.
Fixing the integer $r\geq 0$, our next goal is to find a set of polynomial inequalities
of the form $G_{k,r}(x) \geq G_{k,r}(y)$, or perhaps
the form $G_{k,r}(x) \leq G_{k,r}(y)$,
which would imply (\ref{lem2}). Moreover, let us agree that we want
these polynomials $G_{k,r}(x)$ to have positive coefficients.
Before stating our result for general $r$, let us explain what
it is for the cases $r=0,1,2,3$ in turn.

For $r=0$, we have $\sum_i\psi_0(x_it)
= \sum_i \log(1+x_it)$. But $\log(1+x_it)$ does not have a
Taylor series converging for all values of the variable $t$, so we
cannot use the coefficient of $t^k$ as our choice of polynomial
$G_{k,0}(x)$. A simple remedy is to exponentiate,
obtaining $\prod_i (1+x_it)$, and then use the coefficient of $t^k$
to define the $G_{k,0}(x)$. Then, the hypothesis
$G_{k,0}(x) \geq G_{k,0}(y)$ clearly implies (\ref{lem2}) with $r=0$.
(Here we did not need the assumption $\sum_i x_i =\sum_i y_i$.)

For $r=1$, we have $\sum_i\psi_1(x_it)
= -\sum_i \bigg(\log(1+x_it)-x_it\bigg)$.
In this case, to get an entire function with positive
coefficients, we first negate this, add $(\sum x_i)t$
 and then exponentiate,
obtaining $\prod(1+x_it)$.
We denote the coefficient of $t^k$ by $G_{k,1}(x)$, which happens to
be the same as $G_{k,0}(x)$.
Note that to obtain inequality (\ref{lem2}) for $r=1$, we now need
the hypothesis to be $G_{k,1}(x) \leq G_{k,1}(y)$, because
of the negation performed at the beginning.

For $r=2$, we have $\sum_i\psi_2(x_it)
= \sum_i \bigg(\log(1+x_it)-x_it +\frac{1}{2}x_i^2t^2 \bigg)$.
To end up with positive coefficients after exponentiating,
we again decide to get rid of the negative $\sum-x_it$ terms
by adding on the term $(\sum x_i)t$.
Hence we define $G_{k,2}(x)$ to be the coefficient of $t^k$
in the generating function
$$
\bigg(\prod_{i}(1 + x_it)\bigg)
\exp\bigg(\frac{1}{2}\sum_i x_i^2t^2\bigg),
$$
which we note is entire in $t$, whence its Taylor series
converges to its value for every fixed $x$.
Thus, the conditions $G_{k,2}(x) \geq G_{k,2}(y)$ and
$\sum x_i=\sum y_i$
  imply (\ref{lem2}) with $r=2$.

For $r=3$, we have
$$\sum_i\psi_3(x_it)
= -\sum_i \bigg(\log(1+x_it)-x_it +
\frac{1}{2}x_i^2t^2 -\frac{1}{3}x_i^3t^3 \bigg).$$
As for $r=1$, we first remove the leading $-$ sign.
Then add on the new terms $\big(\sum x_i\big)t
+\frac{1}{3}\big(\sum x_i\big)^3t^3$
to get rid of negative coefficients.
(Note that the polynomial $\big(\sum x_i\big)^3- \sum x_i^3$ has only
positive coefficients.)
Therefore we let $G_{k,3}(x)$ be the coefficient of
$t^k$ in the generating function
$$ \bigg(\prod_{i}(1 + x_it)\bigg)
\exp\bigg(\frac{1}{2}\sum_i x_i^2t^2
+\frac{1}{3}\bigg(\big(\sum_i x_i\big)^3- \sum_i x_i^3\bigg)t^3 \bigg)
\bigg ).
$$
Clearly, the conditions $G_{k,3}(x) \leq G_{k,3}(y)$ and
$\sum x_i=\sum y_i$ imply (\ref{lem2}) with $r=3$.

We continue this pattern for the general case of $r\geq 0$: In the expression
$$(-1)^r\sum_i\psi_r(x_it)=\sum_i \bigg(\log(1+x_it)-Q_r(x_it) \bigg),$$
the terms in $\sum_i -Q_r(x_it)$ with negative coefficients
are precisely those of the form $-\frac{1}{m}\sum_i x_i^mt^m$ for odd $m \leq r$.
So for each such odd $m$
we can add the extra term $ +\frac{1}{m}\bigg(\sum_i x_i\bigg)^mt^m$,
to make all coefficients positive. If $\sum_i x_i =\sum_i y_i$,
then these extra terms are the same with $y_i$ as with $x_i$.
Hence, we can use the exponential
of the resulting  expression as a generating function to define
polynomials $G_{k,r}(x)$ having positive coefficients and
the other desired properties.

We now re-state the latter construction of the $G_{k,r}(x)$ using more formal
notation. For notational efficiency, we observe that the
odd power terms of a polynomial $f(t)$ may be written as $\frac{1}{2}
(f(t)-f(-t))=: f^{-}(t)$.
Thus
$$Q_r^-(t) = \sum_{{\rm odd} \ m \leq r} \frac{1}{m}t^m.$$
Define the generating function $g_r(x,t)$ by
\begin{equation}
\label{gr}
g_r(x,t):=\bigg(\prod_{i=1}^n (1 + x_it)\bigg)
\exp\bigg( - \sum_i Q_r(x_it)+Q_r^-\big(\sum_i x_it\big)\bigg) ,
 \end{equation}
where we recall that $Q_r(s)
:=s-\frac{1}{2}s^2+\dots +(-1)^{r-1}\frac{1}{r}s^r$, the
$r$th Taylor polynomial of $\log(1 + s)$. It is clear from (\ref{gr}) that
$g_r(x,t)$ is an entire function of $t$ for every fixed $x$.

\begin{mydef} For any integers $r\geq 0, k\geq 1$, the
polynomial $G_{k,r}(x)$ is the coefficient of $t^k$ in
the power series expansion of
$g_r(x,t)$.
\end{mydef}
The preceding discussion has shown that the
$G_{k,r}(x)$ are symmetric polynomials with positive coefficients,
and have the following property:
\begin{mythm}
Let $x,y \in [0,\infty)^n$ and fix an integer $r \geq 0. $
Suppose that $\sum_i x_i  = \sum_i y_i$ and that
for all positive integers  $k$,
$$(-1)^r\bigg(G_{k,r}(x)- G_{k,r}(y)\bigg) \geq 0.$$
Then
$||x||_p \geq ||y||_p$ for all real $p$ in the interval $ r \leq p \leq r+1$.
\end{mythm}

In a later
section (\S 5) we will express $G_{k,r}(x)$ in terms of the power sums
$p_m=p_m(x) = \sum_{i=1}^n x_i^m$, $m=1,2,\dots$. To that end, we note
the  following alternative expression for the generating function
$g_r(x,t)$ in the sense of formal power series, which
is easily checked by taking logarithms in (\ref{gr}) and
expanding each $\log(1+x_it)$ in powers of $t$:
\begin{equation}
\label{gr2}
g_r(x,t) = \exp\bigg( \sum_{m\geq 1} (-1)^{m-1}\alpha_mt^m/m
\bigg)
\end{equation}
where $(\alpha_m)_{m=1}^\infty$ is the ``modified" sequence of power sums given by:
\begin{equation}
\label{ag}
 \alpha_m =
\begin{cases}
p_1^m& ,\ \ {\rm for} \ m \ {\rm odd,}\  1\leq m\leq r,\\
0& ,\ \ {\rm for} \  m \ {\rm even,} \  1\leq m\leq r,\\
p_m& ,\ \ {\rm for}\ m\geq r+1.
\end{cases}
\end{equation}
 Equivalently, if $r$ is odd, $(\alpha_m) = (p_1, 0, p_1^3, 0, \dots,
0,p_1^r,p_{r+1}, p_{r+2}, \dots)$, and if $r$ is even, $(\alpha_m) =
(p_1, 0, p_1^3, 0, \dots, p_1^{r-1},0,p_{r+1}, p_{r+2}, \dots)$. Note that
the special cases $r=0$ and $r=1$ give the ``full" sequence of power
sums ($\alpha_m = p_m, \forall m \geq 1 $). Then $g_r(x,t)$ is of course the familiar
generating function of the elementary symmetric polynomials $e_k(x)$.

In addition to the fact that $G_{k,r}(x)$
has positive coefficients, it may be interesting to note
that $\pm G_{k,r}(x)$ is {\it Schur convex} in the sense of
majorization theory \cite{MO}. (The $F_{k,r}(x)$ in our introduction
also have such a property \cite{k2}.)
 Specifically, for $r=0$ and
all odd $r\geq 1$,  $-G_{k,r}(x)$ is Schur convex for all $k$
(equivalently, $G_{k,r}(x)$ is {\it Schur concave}), and
for all even $r\geq 2$, $G_{k,r}(x)$ is Schur convex for all $k$.
As discussed above, for both $r=0,1$ the $G_{k,r}(x)$ are just
the elementary symmetric polynomials, which are well-known
examples of Schur concave functions. The standard proof in the
latter case consists in verifying the Schur-Ostrowski criterion,
which is what we will also do below to prove our assertion
for all $r\geq 1$. In fact, we prove the following slightly stronger
statement involving the generating function $g_r(x,t)$.
\begin{mythm}
Fix any integer $r\geq 1$ and indices $i \neq j$. Then
\begin{equation}
\label{sch}
 (-1)^{r}\bigg(\frac{\partial g_r}{\partial x_i} -
\frac{\partial g_r}{\partial x_j}\bigg)=(x_i -x_j)\gamma(x,t)
\end{equation}
for some function $\gamma(x,t)$ (depending on $r,i,j$ ) having positive coefficients
when expanded as a power series in all of its variables
$x_1, \dots, x_n, t$. In particular, the coefficient
of each $t^k$ in $\gamma(x,t)$ is a polynomial in $x_1, \dots, x_n$
having positive coefficients, and thus each $(-1)^{r}G_{k,r}(x)$
is Schur convex.
\end{mythm}
\begin{proof}
To work out the left-hand side of (\ref{sch}), we may let $i=1$ and $j=2$ by symmetry.
Since $\log g_r = \sum_i \log(1+x_it)
 - \sum_i Q_r(x_it)+Q_r^-\big(\sum_i x_it\big)$, we obtain
$$\frac{1}{g_r}\frac{\partial g_r}{\partial x_1}
=\frac{t}{1+x_1t}  - tQ_r'(x_1t) + t(Q_r^-)'\big(\sum_i x_it\big)$$
$$= t (-1)^r\frac{(x_1t)^{r}}{1+x_1t} + t(Q_r^-)'\big(\sum_i x_it\big)$$
where we have used the geometric series formula
$\frac{1}{1+s}- Q_r'(s)= (-1)^r\frac{s^{r}}{1+s}$,
as seen in the proof of Lemma 1.
Similarly for $x_2$. Hence
$$\frac{1}{g_r}\bigg(\frac{\partial g_r}{\partial x_1} -
\frac{\partial g_r}{\partial x_2}\bigg)
=t^{r+1} (-1)^r\bigg(\frac{x_1^r}{1+x_1t}-\frac{x_2^r}{1+x_2t}\bigg)
$$
$$
=t^{r+1} (-1)^r\frac{(x_1^r-x_2^r)+x_1x_2t(x_1^{r-1}-x_2^{r-1})}{(1+x_1t)(1+x_2t)}.
$$
Multiplying this by $(-1)^rg_r/(x_1-x_2)$, it is now clear that $\gamma(x,t)$
has a power series with positive coefficients in all variables as claimed:
The denominator $(1+x_1t)(1+x_2t)$ will be cancelled by the product $\prod_i(1+x_it)$ in (\ref{gr}),
each of the two terms of type $(x_1^m-x_2^m)/(x_1-x_2)$
simplifies to a sum with positive coefficients, and
the remaining factor $\exp\bigg( - \sum_i Q_r(x_it)+Q_r^-\big(\sum_i x_it\big)\bigg)$
in $g_r$ is also a power series with positive coefficients in all variables (see (\ref{gr})).
\end{proof}
\medskip

\centerline{\bf \S 4. The Polynomials $H_{k,r}(x)$.}
\medskip

In this variant of our topic, we modify the
expression $t - \log(1 + t)$, discussed in \S 2, by simply
replacing $t$ by $t^r$ where $r=2,3, \dots$.

\begin{mylem}
Fix any integer $r\geq 1$ and define
\begin{equation}
\label{phi1}
\phi_r(t) := t^r - \log(1 + t^r).
\end{equation}
 Then $\phi_r(t) > 0$ when
$t>0$, $\phi_r(t) = \mathcal{O}(t^{2r})$
as $t\to 0^+$, and $\phi_r(t) = \mathcal{O}(t^r)$ as $t\to \infty$.
\end{mylem}
The lemma follows from the case $r=1$ of Lemma 1 by substituting
$t^r$ for $t$.
It follows that $a^p$ can be represented by (\ref{r1}) using
$\psi=\phi_r$ whenever $p$ is in the interval $r<p<2r$,
where the corresponding constant $C_p(\phi_r)$ is again positive.
Hence, we deduce:
\begin{mylem}
Let $x,y \in [0,\infty)^n$ and fix an integer $r\geq 1$.
If
\begin{equation}
\label{lem4}
\sum_{i=1}^n \phi_r(x_it)
\geq
\sum_{i=1}^n \phi_r(y_it)
 \end{equation}
for all $t>0$, then
$||x||_p \geq ||y||_p$ for all $p$ in the interval $ r \leq p \leq 2r$.
\end{mylem}
Next, assume the additional hypothesis $\sum_i x_i =\sum_i y_i$.
Fixing the integer $r\geq 1$, we once again wish to obtain
(\ref{lem4}) as a consequence of some stronger family of symmetric polynomial
inequalities, say of the form $H_{k,r}(x) \leq H_{k,r}(y)$ for
some family $H$ having positive coefficients. By the idea
already seen in \S 3, the following generating function $h_r(x,t)$ seems natural
for this purpose:
\begin{equation}
\label{hr}
 h_r(x,t) :=\bigg(\prod_{i}(1 + x_i^rt)\bigg)
\exp\bigg(\bigg(\big(\sum_i x_i\big)^r- \sum_i x_i^r\bigg)t \bigg).
\end{equation}
Note that here we have in effect replaced the $t^r$ by $t$, for
the sake of simplicity of our generating function. As before,
it is clear that $h_r(x,t)$ is entire in $t$ for every $x$.
\begin{mydef} For any integers $r\geq 1, k\geq 1$, the
polynomial $H_{k,r}(x)$ is the coefficient of $t^k$ in
the power series expansion of
$h_r(x,t)$.
\end{mydef}
It is easily seen from the preceding discussion that the
$H_{k,r}(x)$ are symmetric polynomials with positive coefficients,
and have the following property:
\begin{mythm}
Let $x,y \in [0,\infty)^n$ and fix an integer $r \geq 1$.
Suppose that $\sum_i x_i  = \sum_i y_i$ and that
for all positive integers  $k$,
$$H_{k,r}(x) \leq H_{k,r}(y).$$
Then
$||x||_p \geq ||y||_p$ for all real $p$ in the interval $ r \leq p \leq 2r$.
\end{mythm}

As in the previous section \S 3, for $r\geq 1$
  we may express the generating function (\ref{hr}) in terms of
 the power sums $p_m$ using formal power series:
\begin{equation}
\label{hr2}
h_r(x,t) = \exp\bigg( \sum_{m\geq 1} (-1)^{m-1}\alpha_mt^m/m
\bigg)
\end{equation}
where $(\alpha_m)_{m=1}^\infty$ is a modified sequence of power sums given by:
\begin{equation}
\label{ah}
 \alpha_m =
\begin{cases}
p_1^r& ,\ \ {\rm for} \ m=1,\\
p_{mr}& ,\ \ {\rm for} \  m \geq 2.
\end{cases}
\end{equation}
 Equivalently $(\alpha_m) = (p_1^r, p_{2r}, p_{3r}, p_{4r}, \dots)$.

Also as in \S 3, the polynomials $H_{k,r}(x)$ turn out to have a Schur convexity
property
(Schur concavity, in fact), ``inherited" from their generating function:
\begin{mythm}
Fix any integer $r\geq 1$ and indices $i \neq j$. Then
\begin{equation}
\label{schh}
 \frac{\partial h_r}{\partial x_i} -
\frac{\partial h_r}{\partial x_j}=-(x_i -x_j)\delta(x,t)
\end{equation}
for some function $\delta(x,t)$ (depending on $r,i,j$ ) having positive coefficients
when expanded as a power series in all of its variables
$x_1, \dots, x_n, t$. In particular, the coefficient
of each $t^k$ in $\delta(x,t)$ is a polynomial in $x_1, \dots, x_n$
having positive coefficients, and thus each $H_{k,r}(x)$ is
Schur concave.
\end{mythm}
\begin{proof} As in the proof of Theorem 2,
we may let $i=1$ and $j=2$ by symmetry.
Since $\log h_r = \sum_i \log(1+x_i^rt)
 - \sum_i x_i^rt+\big(\sum_i x_i\big)^rt$, we obtain
$$\frac{1}{h_r}\frac{\partial h_r}{\partial x_1}
=\frac{rx_1^{r-1}t}{1+x_1^rt}  - rx_1^{r-1}t + r\big(\sum_i x_i\big)^{r-1}t
=\frac{-rx_1^{2r-1}t^2}{1+x_1^rt} + r\big(\sum_i x_i\big)^{r-1}t.
$$
Similarly for $x_2$. Hence
$$\frac{1}{h_r}\bigg(\frac{\partial h_r}{\partial x_1} -
\frac{\partial h_r}{\partial x_2}\bigg)
=-rt^{2}\bigg(\frac{x_1^{2r-1}}{1+x_1^rt}-\frac{x_2^{2r-1}}{1+x_2^rt}\bigg)
$$
$$
=-rt^{2}\frac{(x_1^{2r-1}-x_2^{2r-1})+x_1^rx_2^rt(x_1^{r-1}-x_2^{r-1})}{(1+x_1^rt)(1+x_2^rt)}.
$$
Multiplying this by $h_r/(x_1-x_2)$, it is clear that $\delta(x,t)$
has a power series with positive coefficients in all variables as claimed,
in view of (\ref{hr}).
\end{proof}
\medskip

\centerline{\bf \S 5. Expressions in terms of power sums.}
\medskip
Our next aim is to express $G_{k,r}(x)$ and $H_{k,r}(x)$ in terms of the power sums
$p_m(x) = \sum_{i=1}^n x_i^m$ where $m$ is a positive integer.
This follows immediately from (\ref{gr2}) and (\ref{hr2}) by well-known
formulas for the exponential of a power series and the (signed) cycle index
polynomials $Z_k$, which we now recall for the reader's convenience in the form
of Lemma 5 bellow.
Let $\alpha_1, \alpha_2, \dots $ be any formal commuting variables or ``indeterminates".
If $\lambda = (\lambda_j)= (\lambda_1 \leq \lambda_2 \leq \dots )$
is a partition of $k$ we define $\alpha_\lambda
= \prod_j \alpha_{\lambda_j}$. Define the sign of $\lambda$ by
${\rm sgn}(\lambda) = \prod_j (-1)^{\lambda_j-1}$ ($ =
\beta_{\lambda}$ for the special sequence $\beta_n = (-1)^{n-1}$).
Let $n_m(\lambda) \geq 0$ denote
the number of $j$'s such that $\lambda_j=m$,
and let $\mathcal{P}_k$ denote
the set of all partitions $\lambda$ of $k$.
If $ \sigma $ is an element of $S_k$, the group of
permutations of $\{1,\dots,k\}$, and if $(\lambda_j)=: \lambda(\sigma) $ are the lengths
of the cycles in the disjoint cycle decomposition of $\sigma$, define
$\alpha_\sigma = \alpha_{\lambda(\sigma)}$. Also, define
$n_m(\sigma)=n_m(\lambda(\sigma))$, and
${\rm sgn}(\sigma) =
{\rm sgn}(\lambda(\sigma))$,
the usual sign of a permutation.
\begin{mylem}{\bf [``Exponential Formula"]}
If $\alpha=(\alpha_m)_{m=1}^\infty$ is any sequence of
indeterminates, then the following identity holds in the sense of
formal power series in $t$.
\begin{equation}
\label{lem5} \exp\bigg( \sum_{m\geq 1} (-1)^{m-1}\alpha_mt^m/m \bigg) =
1+ \sum_{k\geq 1} Z_k(\alpha_1,\dots,\alpha_k)t^k,
 \end{equation}
where the polynomials $Z_k$ are given by:
 \begin{equation}
 \label{lem5aa} Z_k(\alpha_1,\dots,\alpha_k)
=\sum_{\lambda \in \mathcal{P}_k} {\rm sgn}(\lambda)
\prod_{m=1}^k \frac{ (\alpha_{m}/m)^{n_m(\lambda)} }{n_m(\lambda)!}
 \end{equation}
 \begin{equation}
\label{lem5a}
=\frac{1}{k!} \sum_{\sigma \in S_k} {\rm sgn}(\sigma) \alpha_\sigma
 \end{equation}
\begin{equation}
\label{lem5b}
= \frac{1}{k!}\left| \begin{array}{ccccccccc}
            \alpha_1 & 1 & 0 & 0 & \cdot & \cdot & 0\\
            \alpha_2 & \alpha_1 & 2 & 0 & \cdot & \cdot & 0\\
            \alpha_3 & \alpha_2 & \alpha_1 & 3 & \cdot & \cdot & 0\\
            \alpha_4 & \alpha_3 & \alpha_2 & \alpha_1 & \cdot & \cdot & 0\\
            \cdot & \cdot & \cdot & \cdot & \cdot & \cdot & 0\\
             \cdot & \cdot & \cdot & \cdot & \cdot & \cdot & (k-1)\\
            \alpha_k & \alpha_{k-1} & \cdot & \cdot & \cdot & \cdot & \alpha_1\\
          \end{array} \right|.
\end{equation}
\end{mylem}
A convenient reference for this result and further background is
\cite[Theorem 1.3.3]{Stan1}.
The determinant (\ref{lem5b}) is
equivalent to a recursion (or ``Newton identities")
given in \cite[Proposition 5.1.7]{Stan2}.
(One can also recover (\ref{lem5a}) and (\ref{lem5b}) from
\cite[Eqn. (8.30) and (8.31)]{Merris}, which is
the special case $\alpha_m=p_m , \forall m\geq 1$ and $Z_k = e_k ,
\forall k\geq 1$. This implies
the case of arbitrary $\alpha_m$ by the algebraic independence of
$p_1(x),\dots,p_k(x)$ for $n \geq k$.)
\begin{mycor} Fix an integer $r\geq 0$. Then for all $k\geq 1$,
\begin{equation}
\label{cor1}
 G_{k,r} = Z_k(\alpha_1,\dots,\alpha_k)
 =\frac{1}{k!} \sum_{\sigma \in S_k} {\rm sgn}(\sigma) \alpha_\sigma
 \end{equation}
where $(\alpha_m)_{m=1}^\infty$ is the ``modified" sequence of power
sums given by (\ref{ag}). Let the subset $S(k,r) \subset S_k$
consist of those permutations $\sigma$ having no cycles of even length $\leq r$
(equivalently, $n_m(\sigma) = 0$ for all even $m \leq r$). Then (\ref{cor1})
becomes
\begin{equation}
\label{cor1a}
 G_{k,r}
 =\frac{1}{k!} \sum_{\sigma \in S(k,r)} {\rm sgn}(\sigma)
 p_1^{\sum_{j\leq r} jn_j(\sigma) }
 \prod_{j=r+1}^k p_j^{n_j(\sigma)} \ .
 \end{equation}

\end{mycor}
\begin{mycor} Fix an integer $r\geq 1$. Then for all $k\geq 1$,
\begin{equation}
\label{cor2}
 H_{k,r} = Z_k(\alpha_1,\dots,\alpha_k)
 =\frac{1}{k!} \sum_{\sigma \in S_k} {\rm sgn}(\sigma) \alpha_\sigma
 \end{equation}
where $(\alpha_m)_{m=1}^\infty$ is the ``modified" sequence of power
sums given by (\ref{ah}). Substituting (\ref{ah}), this becomes
\begin{equation}
\label{cor2a}
 H_{k,r}
 =\frac{1}{k!} \sum_{\sigma \in S_k} {\rm sgn}(\sigma)
 p_1^{rn_1(\sigma) }
 \prod_{j=2}^k p_{jr}^{n_j(\sigma)} \ .
 \end{equation}
\end{mycor}
 The results (\ref{cor1a}) and (\ref{cor2a}) may  also be expressed
 as follows.
 For each $\sigma \in S(k,r)$,
let $\sigma_r$ denote the product of all cycles of length $\geq r+1$
in the disjoint cycle factorization of $\sigma$. By convention, an empty product
is taken to mean the identity of the group $S_k$. (In terms of the mappings,
 $\sigma_r =\sigma$
on the $\sigma$-orbits of length $\geq r+1$, and  $\sigma_r =$ the identity map on
the $\sigma$-orbits of length $\leq r$.)
Note that $ {\rm sgn}(\sigma_r)={\rm sgn}(\sigma)$ for $\sigma \in S(k,r)$.
Then (\ref{cor1a}) may be written as
 \begin{equation}
\label{cor1b}
G_{k,r}
=\frac{1}{k!}\sum_{\sigma \in S(k,r)}{\rm sgn}(\sigma_r)p_{\sigma_r} \ .
\end{equation}
Regarding  (\ref{cor2a}),
we may alternatively regard $H_{k,r}$ as the coefficient of
$t^{(kr)}$ in the generating function
$$h_r(x,t^r) = \exp \bigg( \sum_{j\geq 1} \alpha_j(-1)^{j-1}t^{jr}/j \bigg).$$
Applying the Exponential Formula (Lemma 5) to the latter and doing
the arithmetic leads to
 \begin{equation}
\label{cor2b}
H_{k,r}
=\frac{(-1)^{k(r-1)}}{(kr)!}\sum_{\sigma \in T(k,r)}
{\rm sgn}(\sigma)
r^{L(\sigma)}p_{\sigma_r} \ ,
\end{equation}
where $T(k,r) \subset S_{kr}$ consists of those permutations $\sigma$ of
$\{1,\dots,kr\}$ whose
disjoint cycles all have lengths divisible by $r$, $L(\sigma)$ denotes
the number of disjoint cycles, and  $\sigma_r$ is again the product
of all cycles of length $> r$
in the disjoint cycle factorization of $\sigma$ (i.e. each $r$-cycle
of $\sigma \in T(k,r)$ is replaced by $r$ $1$-cycles).

{\it Remark.} Fix the integer $r \geq 1$. In the ring of symmetric functions
(in infinitely many variables), we may define an algebraic homomorphism $\phi$
by arbitrarily defining the image of each power sum $p_m$, since these constitute a basis.
Hence we may define a homomorphism by
$\phi(p_m) = \alpha_m$ for all $m\geq 1$, where $\alpha_m$ is given by (\ref{ag}).
Then $\phi(e_k) = G_{k,r}$ for all $k \geq 1$, by Corollary 1. (Since
$\phi(e_k)= \phi(Z_k(p_1,\dots,p_k)) =Z_k(\phi(p_1),\dots,\phi(p_k))
=Z_k(\alpha_1,\dots,\alpha_k) = G_{k,r}$.) Similarly, by Corollary 2, the $H_{k,r}$
are the images of the $e_k$ under the homomorphism $\psi$
defined by $p_m \mapsto \alpha_m$
where now $\alpha_m$ means (\ref{ah}). We digress to mention the following
question regarding $\psi$. Let $\mathcal{S}$ denote the set of all linear combinations,
with positive real coefficients, of all Schur functions (see
\cite[\S 4.4]{Sagan}). Is $\psi(\mathcal{S}) \subset \mathcal{S}$ ?
(In particular, is each $H_{k,r}$
a linear combination,
with positive real coefficients, of  Schur functions ?)

\medskip

\centerline{\bf \S 6. Expressions in terms of matrix entries.}
\medskip
Let $A=[a_{ij}]$ be a complex $n \times n$ matrix, and let $(x_1,\dots, x_n) = x$ be
its eigenvalues. In this section we briefly consider the question of how to express
$G_{k,r}(x)$ and $H_{k,r}(x)$ as polynomials in the entries $a_{ij}$ of $A$.
This is possible of course for {\it any} symmetric polynomial $F(x)$, since
by a fundamental result $F(x)$ may first be written  as a polynomial in the
power sums $p_k(x)$ (or, if we prefer, in the elementary symmetric
polynomials $e_k(x)$), and these in turn have well-known
polynomial expressions in terms of the entries $a_{ij}$;
$ p_k(x) = {\rm Trace}(A^k) \ :=: \ <A^k>$ and $e_k(x)=$ sum of all principal $k \times k$
subdeterminants of $A$. We can therefore already give one answer quite
explicitly using the polynomials $Z_k(\alpha)$ from Corollaries 1 or 2 of the previous
section, replacing each occurrence of a $p_m$ by the trace $<A^m>$. For example
we thus obtain:
$$G_{5,3}(x) = Z_5(p_1,0,p_1^3,p_4,p_5) =  Z_5(<A>,0,<A>^3,<A^4>,<A^5>)$$
%

$\displaystyle{=\frac{1}{5!} \sum_{\sigma \in S(5,3)} {\rm sgn}(\sigma)
<A>^{n_1(\sigma)+3n_3(\sigma)}<A^4>^{n_4(\sigma)}<A^5>^{n_5(\sigma)} }
$
\medskip

$\displaystyle{=\frac{7}{40}<A>^5 - \frac{1}{4}<A><A^4> + \frac{1}{5}<A^5> .}
$
\medskip

\noindent
However, we will now present another kind of expression which uses more explicitly
the individual monomials (products) of the entries $a_{ij}$ of $A$.
For example, $G_{k,1}(x) = e_k(x)$ is on the one hand given by
$Z_k(<A>,<A^2>,\dots, <A^k>)$, but on the other hand $e_k(x)$ is also equal to
the sum of all principal $k \times k$ subdeterminants of $A$, which
can immediately be written in terms of products of $a_{ij}$ using the familiar
expansion of a determinant as a sum over permutations.
We thus aim to generalize this latter type of expansion
for any $G_{k,r}$ or $H_{k,r}$, and will do so essentially by
imitating what happens in the special case of $e_k =G_{k,1}$.
We need some of the machinery of representation
theory. Following the sketch given
in \cite[Ch. 7, Appendix 2, pp. 444-445]{Stan2}, or the more detailed
account \cite[Ch. 6]{Merris},
identify $A$ with the linear map $A:V\to V$ of the vector space $V=\C^n$
(of column vectors) as usual, and consider the $k$-fold
tensor product $V^{\otimes k} = V\otimes \dots \otimes V$.
There is an action of $A$ which we denote by $A^{\otimes k}$
(the ``Kronecker power")
from $V^{\otimes k}$ to itself,
characterized by $A^{\otimes k}(v_1 \otimes \dots \otimes v_k) =
(Av_1) \otimes \dots \otimes (Av_k)$
for all elements of the form $v_1 \otimes \dots \otimes v_k$.
Also, there is an action for any $\sigma \in S_k$, characterized by
$\sigma(v_1 \otimes \dots \otimes v_k) = v_{\sigma^{-1}(1)}
\otimes \dots \otimes v_{\sigma^{-1}(k)}$. The two actions
commute; $\sigma A^{\otimes k} = A^{\otimes k}\sigma$ on $V^{\otimes k}$. The trace
of $A^{\otimes k}\sigma$ on $V^{\otimes k}$ depends only on the conjugacy
class of $\sigma$ and eigenvalues $x$ of $A$:
\begin{equation}
\label{tr1}
{\rm Trace} (A^{\otimes k}\sigma) = p_\sigma(x) = \prod_j p_{\lambda_j(\sigma)}(x),
\end{equation}
where the $p_m$ are the power sum polynomials as before.
(This can be seen by using a basis of $V$ which makes $A$
triangular.) By linearity, for any  sequence
of permutations $\sigma_s \in S_k$ and constants $c_s \in \C$
on some finite index set $S$ we have
\begin{equation}
\label{tr2}
{\rm Trace} \bigg(\sum_{s\in S}
c_sA^{\otimes k}\sigma_s \bigg) = \sum_{s \in S}c_s p_{\sigma_s}(x).
\end{equation}
On the other hand, we may also compute a trace
on $V^{\otimes k}$ using the basis $\{u_{m(1)}
\otimes \dots \otimes u_{m(k)} \}$
where $\{u_j\}_{j=1}^n$ is the standard basis of $V=\C^n$,
and $m \in \Gamma_{k,n} =$
set of all functions $m:\{1,\dots,k\} \to \{1,\dots,n\}$.
This will yield the above trace (\ref{tr2}) directly in terms of products of
entries $a_{ij}$ of $A$.
This computation is well-known, but we will now reproduce it here for convenience.
We will essentially follow \cite[Ch. 6 and 7]{Merris}. For a given pair
of sequences $(c_s, \sigma_s)_{s\in S} =:C$ as above, define
a ``matrix function" $d_C$ on
complex $k \times k$ matrices $B=[B_{i,j}]$ by
\begin{equation}
\label{dc}
d_C(B) = \sum_{s\in S}c_s\prod_{j=1}^k B_{\sigma_s(j),j}
\end{equation}
 Now consider a fixed  $m \in \Gamma_{k,n}$.
Recalling that $u_j$ is the $j$th standard basis column vector
in $\C^k$, denote
$$u(m) := u_{m(1)}\otimes u_{m(2)}\otimes \dots \otimes u_{m(k)}.$$
Denote the entries of $A$ by $A(i,j):=a_{i,j}$
for the sake of better legibility.
Let $\langle v, w_j\rangle$ denote the coefficient of $w_j$ in $v$ whenever
$v$ is an element of a vector space of which $\{w_j\}$ is a basis.
Then for each $\sigma \in S_k$ we find that
$$  \bigg< A^{\otimes k}\sigma
 \bigg(u(m)\bigg)\ , \ u(m) \bigg>
=\prod_{j=1}^k A\bigg(m(\sigma(j)), m(j)\bigg)
=:\prod_{j=1}^k A[m|m]_{\sigma(j),j}
$$
 where $A[m|m]$ denotes the
$k \times k$ matrix whose $(i,j)$ entry is $A(m(i), m(j))$.
We remark that $A[m|m]$ is a (principal) submatrix of $A$ only when
the function $m$ is strictly increasing, otherwise it may be thought
of as a ``submatrix" which allows repetition and permutation
of the original indices. Hence, for each fixed $m \in \Gamma_{k,n} $,
$$ \sum_{s\in S}c_s \bigg< A^{\otimes k}\sigma_s
 \bigg(u(m)\bigg)\ , \ u(m) \bigg>
 =d_C(A[m|m]).$$
 Summing over all $m$ gives
 $$
 {\rm Trace} \bigg(\sum_{s\in S}
c_sA^{\otimes k}\sigma_s \bigg)  = \sum_{m\in \Gamma_{k,n} } d_C(A[m|m]).$$
Hence by (\ref{tr2}),
\begin{equation}
\label{tr3}
  \sum_{s\in S}c_s p_{\sigma_s}(x)=\sum_{m\in \Gamma_{k,n} } d_C(A[m|m]) .
\end{equation}
We call $C = (c_s,\sigma_s)_{s\in S}$ a {\it class function}
on $S_k$ if the index set $S=S_k$, $\sigma_s = s$,
and $c_s$ depends only on the conjugacy class of $s$.
When $C$ is a class function,
the latter sum may be simplified further by noting that
$d_C(A[m|m])=d_C(A[m'|m'])$
whenever $m$ and $m'$ are permutations of each other, i.e. when
$m' = m\circ \tau$ for some $\tau \in S_k$. Let $(N_1,\dots, N_n) =:N$
denote an $n$-tuple of integers $N_i \geq 0$ with $\sum_i N_i = k$,
and let $m^{(N)} \in \Gamma_{k,n}$ denote the unique nondecreasing
function which takes on each value $i$ exactly $N_i$ times. Then
there are exactly $k!/(N_1! \dots N_n!)$ distinct permutations
of $m^{(N)}$ in $\Gamma_{k,n}$. Hence, whenever a matrix function
$d_\gamma$ can be shown to be of the form
$d_\gamma = d_C$ for some class function $C$, then we have
$$ \sum_{m\in \Gamma_{k,n} } d_\gamma(A[m|m])
=\sum_{\sum_i N_i = k }
d_\gamma(A[m^{(N)}|m^{(N)}])\frac{k!}{N_1! \dots N_n!}.$$

For our cases of interest, it remains
to apply (\ref{tr3}) to the results (\ref{cor1b})
or (\ref{cor2b}).
In the first case, one obtains
 \begin{equation}
\label{gA}
G_{k,r}(x) = \frac{1}{k!}\sum_{m\in \Gamma_{k,n} } \delta_r(A[m|m])
=\sum_{\sum_i N_i = k } \delta_r(A[m^{(N)}|m^{(N)}])\frac{1}{N_1! \dots N_n!},
\end{equation}
where $\delta_r(B)$ is a ``modified determinant" defined
 for any $k \times k$ matrix $B$ by
 \begin{equation}
\label{dr}
\delta_r(B) =\sum_{\sigma \in S(k,r)}
{\rm sgn}(\sigma_r)\prod_{j=1}^k B_{\sigma_r(j),j} \ ,
\end{equation}
which is easily seen to be of the form $d_C(B)$ for some class function $C$
on $S_k$.
 Similarly, from (\ref{cor2b}) one obtains
\begin{equation}
\label{hA}
H_{k,r}(x) = \frac{1}{(kr)!}\sum_{m\in \Gamma_{kr,n} } D_r(A[m|m])
=\sum_{\sum_i N_i = kr } D_r(A[m^{(N)}|m^{(N)}])\frac{1}{N_1! \dots N_n!},
\end{equation}
where $D_r(B)$ is defined for any $kr \times kr$ matrix $B$ as:
 \begin{equation}
\label{Dr}
D_r(B) =(-1)^{k(r-1)}\sum_{\sigma \in T(k,r)}
{\rm sgn}(\sigma)r^{L(\sigma)}\prod_{j=1}^{kr} B_{\sigma_r(j),j} \ ,
\end{equation}
which can be checked to be $d_C(B)$ for some class function $C$
on $S_{kr}$.



\end{document}